\documentclass[12pt]{amsart}

\usepackage{amsmath}
\usepackage{amsfonts}
\usepackage{amssymb}
\usepackage{amsthm}
\usepackage{hyperref}
\usepackage{graphicx}
\usepackage{enumerate}
\usepackage{cleveref}
\usepackage{color}
\usepackage{mathrsfs}
\usepackage{geometry}

\newtheorem{theorem}{Theorem}[section]
\newtheorem{lemma}[theorem]{Lemma}
\newtheorem{proposition}[theorem]{Proposition}
\newtheorem{corollary}[theorem]{Corollary}
\newtheorem{definition}[theorem]{Definition}
\newtheorem{remark}[theorem]{Remark}
\newtheorem{conjecture}[theorem]{Conjecture}
\Crefname{conjecture}{Conjecture}{Conjectures}

\theoremstyle{plain}

\theoremstyle{plain}

\newcommand{\N}{\mathbb{N}}

\newcommand{\Z}{\mathbb{Z}}

\newcommand{\R}{\mathbb{R}}
\newcommand{\C}{\mathbb{C}}

\newcommand{\bbH}{\mathbb{H}}

\newcommand{\mymod}{\operatorname{mod}}

\newcommand{\eps}{\varepsilon}

\newcommand{\SL}{\operatorname{SL}}

\newcommand{\calH}{\mathscr{H}}
\newcommand{\G}{\mathscr{G}}

\newcommand{\sign}{\operatorname{sgn}}
\renewcommand{\Re}{\operatorname{Re}}
\renewcommand{\Im}{\operatorname{Im}}
\newcommand{\calR}{\mathscr{R}}

\numberwithin{equation}{section}

\author{Michael H. Mertens}
\title{Mock Modular Forms and Class Number Relations}
\address{Mathematisches Institut der Universit\"at zu K\"oln \\ 
         Weyertal 86-90 \\
         D-50931 K\"oln, Germany}
\email{mmertens@math.uni-koeln.de}
\urladdr{http://www.mi.uni-koeln.de/~mmertens}
\thanks{The author's research is supported by the DFG-Graduiertenkolleg 1269 "Global Structures in Geometry and Analysis".\\
This paper is part of the author's PhD-thesis, written under the supervision of Prof. Dr. K. Bringmann at the Universit\"at zu K\"oln.}

\begin{document}
\begin{abstract}
In this paper, we prove an almost 40 year old conjecture by H. Cohen concerning the generating function of the Hurwitz class number of quadratic forms using the theory of mock modular forms. This conjecture yields an infinite number of so far unproven class number relations. 
\end{abstract}

\maketitle
\section{Introduction}
Since the days of C.F. Gau\ss, it has been an important problem in number theory to determine the class numbers of binary quadratic forms. One aspect of this, which is also of interest regarding computational issues, are the so called \emph{class number relations}. These express certain sums of class numbers in terms of more elementary arithmetic functions which are easier to understand and computationally more feasible. The first examples of these relations are due to L. Kronecker \cite{Kro60} and A. Hurwitz \cite{Hur84}\cite{Hur85}:

Let $H(n)$ denotes the Hurwitz class number of a non-negative integer $n$ (cf. \Cref{secBasics} for the definition). Then we have the relation
\begin{equation}\label{KroHur}
\sum\limits_{s\in\Z}H(4n-s^2)+2\lambda_1(n)=2\sigma_1(n),
\end{equation}
where 
\begin{equation}\label{lambda}
\lambda_k(n):=\frac 12\sum\limits_{d\mid n}\min\left(d,\frac nd\right)^k
\end{equation} 
and $\sigma_k(n):=\sum_{d\mid n} d^k$ is the usual $k$-th power divisor sum.

This was further extended by M. Eichler in \cite{Eichler55}. For odd $n\in\N$ we have
\begin{equation}\label{Eic}
\sum\limits_{s\in\Z}H(n-s^2)+\lambda_1(n)=\frac 13\sigma_1(n).
\end{equation}
Other such examples of class number relations can be obtained, e.g. from the famous Eichler-Selberg trace formula for cusp forms on $\SL_2(\Z)$.

In 1975, H. Cohen \cite{Coh75} generalized the Hurwitz class number using Dirichlet's class number formula (see e.g. \cite{Dav}) to a number $H(r,n)$ which is closely related to the value of a certain Dirichlet $L$-series at $(1-r)$ and showed that for $r\geq 2$ the generating function
\[\calH_r(\tau):=\sum\limits_{n=0}^\infty H(r,n)q^n,\quad q=e^{2\pi i\tau},\quad \Im(\tau)>0\] 
is a modular form of weight $r+\tfrac 12$ on $\Gamma_0(4)$ (\cite[Theorem 3.1]{Coh75}). This yields many interesting relations in the shape of \eqref{KroHur} and \eqref{Eic} for $H(r,n)$.

The case $r=1$, where $H(1,n)=H(n)$, was treated around the same time by D. Zagier (\cite{Zagier75},\cite[Chapter 2]{HZ76}): He showed that the function $\calH(\tau)=\calH_1(\tau)$ is in fact \emph{not} a modular form but can be completed by a non-holomorphic term such that the completed function transforms like a modular form of weight $\tfrac 32$ on $\Gamma_0(4)$. 

\mbox{}\\
In more recent years, this phenomenon has been understood in a broader context: The discovery of the theory behind Ramanujan's mock theta functions by S. Zwegers \cite{ZwegersDiss}, J.H. Bruinier and J. Funke \cite{BF04}, K. Bringmann and K. Ono \cite{BO06} and many, many others has revealed that the function $\calH$ is an example of a weight $\tfrac 32$ mock modular form, i.e. the holomorphic part of a harmonic weak Maa{\ss} form\footnote{In the literature the spelling ``Maass form'' is more common, although these functions are named after the German mathematician Hans Maa{\ss} (1911-1992)} (see \Cref{secBasics} for a definition). Using this theory, some quite unexpected connections to combinatorics occur, as for example in \cite{BL09}, where class numbers were related to ranks of so-called overpartitions. 

\mbox{}\\
In \cite{Coh75}, Cohen considered the formal power series
\begin{equation}\label{eqCohen}
S_4^1(\tau,X):=\sum\limits_{\substack{n=0\\ n\text{ odd}}}^\infty \left[\sum_{\substack{s\in\Z\\ s^2\leq n}}\frac{H(n-s^2)}{1-2sX+nX^2}+\sum\limits_{k=0}^\infty \lambda_{2k+1}(n)X^{2k}\right]q^n.
\end{equation}
From Zagier's and his own results, as well as computer calculations, he conjectured that the following should be true.
\begin{conjecture}\label{conjCohen}(H. Cohen, 1975)\\
The coefficient of $X^\ell$ in the formal power series in \eqref{eqCohen} is a (holomorphic) modular form of weight $\ell+2$ on $\Gamma_0(4)$.
\end{conjecture}

The goal of this paper is to prove the following result.
\begin{theorem}\label{main}
\Cref{conjCohen} is true. Moreover, for $\ell>0$ the coefficient of $X^\ell$ in \eqref{eqCohen} is a cusp form. 
\end{theorem}
This obviously implies new relations for Hurwitz class numbers which to the author's knowledge have not been proven so far. We give some of them explicitly in \Cref{corClass}.

The main idea of the proof of \Cref{main} is to relate both summands in the coefficient of the above power series to objects which in accordance to the nomenclature in \cite{Zag12} should be called \emph{quasi mixed mock modular forms}, complete them, such that they transform like modular forms and show that the completion terms cancel each other out. The same idea is also used in a very recent preprint \cite{BK13} by K. Bringmann and B. Kane in which they also prove several identities for sums of Hurwitz class numbers conjectured by B. Brown et al. in \cite{Betal}.

The outline of this paper is as follows: the preliminaries and notations are explained in \Cref{secBasics}. \Cref{secId} contains some useful identities and other lemmas which will be used in \Cref{secProof} to prove Cohen's conjecture.

Since many of the proofs involve rather long calculations, we omit some of them here. More detailed proofs will be available in the author's PhD thesis \cite{Mert14}.
\section{Preliminaries}\label{secBasics}
First, we fix some notation. For this entire paper, let $\tau$ be a variable from the complex upper half plane $\bbH$ and denote $x:=\Re(\tau)$ and $y:=\Im(\tau)$. As usual, we set $q:=e^{2\pi i\tau}$. The letters $u,v$ denote arbitrary complex variables. The differential operators with respect to all these variables shall be renormalized by a factor of $\tfrac{1}{2\pi i}$, thus we abbreviate
\[D_t:=\frac {1}{2\pi i}\frac{d}{d t}.\]

An element of $\SL_2(\Z)$ is always denoted by $\gamma=\left(\begin{smallmatrix} a & b \\ c & d \end{smallmatrix}\right)$. For some natural number $N$, we set as  usual 
\[\Gamma_0(N)=\left\lbrace \gamma=\begin{pmatrix} a & b \\ c & d\end{pmatrix}\in\SL_2(\Z) \mid c\equiv 0\quad(\mymod N)\right\rbrace.\]

There are two different theta series occuring in this paper. One is the theta series of the lattice $\Z$,
\begin{equation}\label{theta}
\vartheta(\tau):=\sum\limits_{n\in\Z} q^{n^2},
\end{equation}
while the other is the classical Jacobi theta series
\begin{equation}\label{Theta}
\Theta(v;\tau):=\sum\limits_{\nu\in\frac 12 +\Z} q^{\nu^2/2}e^{2\pi i\nu(v+1/2)}.
\end{equation}
Note that, e.g. in \cite{ZwegersDiss}, the letter $\vartheta$ stands for the Jacobi theta series.
\subsection{Mock Modular Forms}
In this subsection we give some basic definitions about harmonic Maa{\ss} forms and mock modular forms (for details, cf. \cite{BF04,Ono} et al.). Therefore we fix some $k\in\tfrac 12\Z$ and define for a smooth function $f:\bbH\rightarrow \C$ the following operators:
\begin{enumerate}
\item The weight $k$ \emph{slash operator} by
\[(f|_k\gamma)(\tau)=\begin{cases} (c\tau+d)^{-k}f\left(\frac{a\tau+b}{c\tau+d}\right) &,\,\text{if }k\in\Z \\
                                      \left(\frac{c}{d}\right)\eps_d(c\tau+d)^{-k}f\left(\frac{a\tau+b}{c\tau+d}\right) &,\,\text{if }k\in\tfrac 12+\Z
                                      \end{cases},\]
where $\left(\frac mn\right)$ denotes the extended Legendre symbol in the sense of \cite{Shi}, $\tau^{1/2}$ denotes the principal branch of the square root (i.e. $-\tfrac\pi 2<\arg(\tau^{1/2})\leq\tfrac \pi 2$), and
\[\eps_d:=\begin{cases} 1 & \text{, if }d\equiv 1\quad (\mymod 4)\\ i & \text{, if }d\equiv 1\quad (\mymod 4).\end{cases}\]
We shall assume $\gamma\in\Gamma_0(4)$ if $k\not\in\Z$.
\item The weight $k$ \emph{hyperbolic Laplacian} by ($\tau=x+iy$)
\[\Delta_k:=-y^2\left(\frac{\partial^2}{\partial x^2}+\frac{\partial^2}{\partial y^2}\right)+iky\left(\frac{\partial}{\partial x}+i\frac{\partial}{\partial y}\right).\]
\end{enumerate}
\begin{definition}
Let $f:\bbH\rightarrow\C$ be a smooth function and $k\in\tfrac 12\Z$. We call $f$ a \emph{harmonic weak Maa{\ss} form} of weight $k$ on some subgroup $\Gamma\leq\SL_2(\Z)$ (resp. $\Gamma_0(4)$ if $k\not\in\Z$) of finite index if the following conditions are met:
\begin{enumerate}
\item $(f|_k\gamma)(\tau)=f(\tau)$ for all $\gamma\in\Gamma$ and $\tau\in\bbH$
\item $(\Delta_kf)(\tau)=0$ for all $\tau\in\bbH$
\item $f$ grows at most exponentially in the cusps of $\Gamma$. 
\end{enumerate}
\end{definition}
\begin{proposition}(\cite[Lemma 7.2, equation (7.8)]{Ono})\\
Let $f$ be a harmonic Maa{\ss} form of weight $k$ with $k>0$ and $k\neq 1$. Then there is canonical splitting
\[f(\tau)=f^+(\tau) + f^-(\tau),\]
where for some $N\in\Z$ we have
\[f^+(\tau)=\sum\limits_{n=N}^\infty c_f^+(n)q^n\]
\[f^-(\tau)=c_f^-(0)\frac{(4\pi y)^{-k+1}}{k-1}+ \sum\limits_{n=1}^\infty c_f^-(-n)\Gamma(k-1;4\pi ny)q^{-n}.\]
Here 
\[\Gamma(\alpha;x)=\int\limits_x^\infty e^{-t}t^{\alpha-1}d t\]
is the incomplete Gamma function.
\end{proposition}
\begin{definition}
\begin{enumerate}[(i)]
\item The functions $f^+$ and $f^-$ in the above proposition are referred to as the \emph{holomorphic} and \emph{non-holomorphic part} of the harmonic weak Maa{\ss} form $f$.
\item A \emph{mock modular form} is the holomorphic part of a harmonic weak Maa{\ss} form.
\end{enumerate}
\end{definition}
There are several generalizations of mock modular forms, e.g. \emph{mixed mock modular forms}, which are essentially products of mock modular forms and usual holomorphic modular forms. For details, we refer the reader to \cite[Section 7.3]{Zag12}. 
\subsection{Class Numbers}
Let $d$ be a non-negative integer with $d\equiv 0,3\quad (\mymod 4)$. Then the \emph{class number} for the discriminant $-d$ is the number of $\SL_2(\Z)$ equivalence classes of primitive binary integral quadratic forms of discriminant $-d$,
\begin{equation}
h(-d)=\left\vert\left\{Q=\begin{pmatrix} 2a & b \\ b & 2c \end{pmatrix}\in\Z^{2\times 2}\,\mid\,\det Q=d\text{ and }\gcd(a,b,c)=1 \right\}/\SL_2(\Z)\right\vert,
\end{equation}
where of course $\SL_2(\Z)$ acts via $(Q,\gamma)\mapsto \gamma^{tr}Q\gamma$.

The Hurwitz class number is now a weighted sum of these class numbers: Define $w_3=3$, $w_4=2$ and $w_d=1$ for $d\neq 3,4$. Then the Hurwitz class number is given by
\[
H(n)=\begin{cases} -\frac{1}{12} &\text{if }n=0,\\
                     0             &\text{if }n \equiv 1,2\quad (\mymod 4),\\
                     \sum\limits_{f^2\mid n} \frac{h(-n/f^2)}{w_{n/f^2}} &\text{ otherwise}. \end{cases}
                     \]
The generating function of the Hurwitz class number shall be denoted by
\[\calH(\tau)=\sum\limits_{n=0}^\infty H(n)q^n.\]
We have the following result concerning a modular completion of the function $\calH$ which was already mentioned in the introduction (cf. \cite[Chapter 2, Theorem 2]{HZ76}). 
\begin{theorem}\label{ModcalH}
Let 
\[\calR(\tau)=\frac{1+i}{16\pi}\int\limits_{-\overline{\tau}}^{i\infty}\frac{\vartheta(z)}{(z+\tau)^{3/2}}d z.\]
Then the function
\[\G(\tau)=\calH(\tau)+\calR(\tau)\]
transforms under $\Gamma_0(4)$ like a modular form of weight $\tfrac 32$.
\end{theorem}
The idea of the proof is to write $\calH$ as a linear combination of Eisenstein series of weight $\tfrac 32$, in analogy to the proof of \cite[Theorem 3.1]{Coh75}. These series diverge, but using an idea of Hecke (cf. \cite[\S 2]{Hecke}), who used it to derive the transformation law of the weight $2$ Eisenstein series $E_2$, one finds the non-holomorphic completion term $\calR$.

\mbox{}\\
It is easy to check that $\G$ is indeed a harmonic weak Maa{\ss} form of weight $\tfrac 32$. As a mock modular form, the function $\calH$ is rather peculiar since it is basically the only example of such an object which  is holomorphic at the cusps of $\Gamma_0(4)$ (cf. \cite[Section 7]{Zag12}).
\subsection{Appell-Lerch Sums}\label{secAppell}
In this subsection we are going to recall some general facts about Appell-Lerch sums. For details, we refer the reader to \cite{Zwegers10,ZwegersDiss}.
\begin{definition}
Let $\tau\in\bbH$ and $u,v\in\C\setminus(\Z\tau+\Z)$. The \emph{Appell-Lerch sum} of level $1$ is then the following expression:
\[A_1(u,v)=A_1(u,v;\tau):=a^{1/2}\sum\limits_{n\in\Z} \frac{(-1)^{n} q^{n(n+1)/2}b^n}{1-aq^n}\]
where $a=e^{2\pi i u},\,b=e^{2\pi i v},$ and $q=e^{2\pi i \tau}$.
\end{definition}
In addition, we define the following real-analytic functions.
\begin{align}
\label{R} R(u;\tau)&:=\sum\limits_{\nu\in\tfrac 12+\Z}\left\lbrace \sign(\nu)-E\left(\left(\nu+\frac{\Im u}{y}\right)\sqrt{2y} \right) \right\rbrace (-1)^{\nu -1/2}q^{-\nu^2/2}e^{-2\pi i\nu u},\\
\label{E} E(t)&:=2\int\limits_0^t e^{-\pi u^2}d u=\sign(t)\left(1-\beta\left(t^2\right)\right), \\
\label{beta} \beta(x)&:=\int\limits_x^\infty u^{-1/2}e^{-\pi u}d u,
\end{align}
where for the second equality in \eqref{E} we refer to \cite[Lemma 1.7]{ZwegersDiss}.

This function $R$ has some nice properties, a few of which are collected in the following Propositions.
\begin{proposition}(cf. \cite[Proposition 1.9]{ZwegersDiss})\\\label{Rfunc}
The function $R$ fulfills the elliptic transformation properties
\begin{enumerate}[(i)]
\item $R(u+1;\tau)=-R(u;\tau)$
\item $R(u;\tau)+e^{-2\pi iu-\pi i\tau}R(u+\tau;\tau)=2e^{-\pi iu-\pi i\tau/4}$
\item $R(-u)=R(u)$.
\end{enumerate}
\end{proposition}
The following proposition has already been mentioned in \cite{BZ10}. The proof is a straightforward computation.
\begin{proposition}\label{Heat}
The function $R$ lies in the kernel of the renormalized Heat operator $2D_\tau+D_u^2$, hence
\begin{equation}\label{eqHeat}
D_u^2R=-2D_\tau R.
\end{equation}
\end{proposition}

We now define the non-holomorphic function
\[\widehat{A}_1(u,v;\tau)=A_1(u,v;\tau)+\frac{i}{2}\Theta(v;\tau)R(u-v;\tau)\]
which will henceforth be referred to as the \emph{completion} of the Appell-Lerch sum $A_1$.
\begin{theorem}(\cite[Theorem 2.2]{Zwegers10})\\
The real-analytic function $\widehat{A}_1$ transforms like a Jacobi form of weight $1$ and index $\begin{pmatrix} -1 & 1 \\ 1 & 0 \end{pmatrix}$:
\begin{enumerate}
\item[(S)] $\widehat{A}_1(-u,-v)=-\widehat{A}_1(u,v)$.
\item[(E)] $\widehat{A}_1(u+\lambda_1\tau+\mu_1,v+\lambda_2\tau+\mu_2)=(-1)^{(\lambda_1+\mu_1)}a^{\lambda_1-\lambda_2}b^{-\lambda_1}
q^{\lambda_1^2/2-\lambda_1\lambda_2}\widehat{A}_1(u,v)$ for $\lambda_i,\mu_i\in\Z$.
\item[(M)] $\widehat{A}_1(\frac{u}{c\tau+d},\frac{v}{c\tau+d};\frac{a\tau+b}{c\tau+d})=(c\tau+d)e^{\pi i c(-u^2+2uv)/(c\tau+d)}\widehat{A}_1(u,v;\tau)$\\
for $\gamma=\begin{pmatrix} a & b \\ c & d \end{pmatrix}\in\SL_2(\Z)$.
\end{enumerate}
\end{theorem}
\section{Some Lemmas}\label{secId}
As we mentioned in the introduction, we would like to relate the two summands for each coefficient in the power series in \Cref{conjCohen} to some sort of modular object. For that purpose, we recall the definition of Rankin-Cohen brackets as given in \cite[p. 53]{MF123}, which differs slightly (see below) from the original one in \cite[Theorem 7.1]{Coh75}.
\begin{definition}
Let $f,g$ be smooth functions defined on the upper half plane and $k,\ell\in\R_{>0}$, $n\in\N_0$. Then we define the $n$-th \emph{Rankin-Cohen bracket} of $f$ and $g$ as
\[[f,g]_n=\sum\limits_{r+s=n} (-1)^r {{k+n-1} \choose s} {{\ell+n-1} \choose r} D^rfD^sg\]
where for non-integral entries we define
\[{m \choose s}:=\frac{\Gamma(m+1)}{\Gamma(s+1)\Gamma(m-s+1)}.\]
Here, the letter $\Gamma$ denotes the usual Gamma function. 
\end{definition}
It is well-known (cf. \cite[Theorem 7.1]{Coh75}) that if $f,g$ transform like modular forms of weight $k$ and $\ell$ resp. then $[f,g]_n$ transforms like a modular form of weight $k+\ell+2n$ and that $[f,g]_0=f\cdot g$. The interaction of the first Rankin-Cohen bracket, which itself fulfills the Jacobi identity of Lie brackets, and the regular product of modular forms give the graded algebra of modular forms the additional structure of a so-called Poisson algebra (cf. \cite[p. 53]{MF123}).

Note that our definition of the Rankin-Cohen bracket differs from the one in \cite[Theorem 7.1]{Coh75} by a factor of $n!(-2\pi i)^n$ which guarantees that if $f$ and $g$ have integer Fourier coefficients, so does $[f,g]_n$.

\begin{lemma}\label{conjMod}
The coefficient of $X^{2k}$ in \eqref{eqCohen} is given by
\begin{equation}\label{eqCohen2}
\frac{c_k}{2}\left([\calH,\vartheta]_k(\tau)-[\calH,\vartheta]_k\left(\tau+\tfrac 12\right)\right)+\Lambda_{2k+1,odd}(\tau),
\end{equation}
where $c_k=k!\frac{\sqrt{\pi}}{\Gamma(k+\tfrac 12)}$ and 
\begin{equation}
\Lambda_{\ell,odd}(\tau):=\sum\limits_{n=0}^\infty \lambda_\ell(2n+1)q^{2n+1}
\end{equation}
with $\lambda_\ell$ as in \eqref{lambda}. The coefficient of $X^{2k+1}$ is identically $0$.
\end{lemma}
\begin{proof}
As in \cite[Theorem 6.1]{Coh75} we define for a modular form $f$ with $f(\tau)=\sum_{n=0}^\infty a(n)q^n$ of weight $k$ and an integer $D\neq 0$ the series
\[S_D^f(\tau;X):=\sum\limits_{n=0}^\infty \left(\sum\limits_{\substack{s\in\Z\\ s^2\leq n}}\frac{a\left(\frac{n-s^2}{\vert D\vert}\right)}{(1-sX+nX^2)^{k-\tfrac 12}}\right)q^n,\]
where we assume $a(n)=0$ if $n\notin\N_0$.
From p. 283 in \cite{Coh75} we immediately get the formula 
\[S_D^f(\tau;X)=\sum\limits_{n=0}^\infty \frac{\sqrt{\pi}\Gamma\left(n+k-\tfrac 12\right)}{\Gamma\left(n+\tfrac 12\right)\Gamma\left(k-\tfrac 12\right)}[g,\vartheta]_n(\tau) ,\]
where $g(\tau):=f(\vert D\vert\tau)$. This yields the assertion by plugging in $f=\calH$, $k=\tfrac 32$, and $D=1$.
\end{proof}
Since in the Rankin-Cohen brackets that we consider here, we have linear combinations of products of derivatives of a mock modular form and a regular modular form, one could call an object like this a \emph{quasi mixed mock modular form.}
\begin{lemma}\label{lemAppell}
For odd $k\in\N$, the function $\Lambda_{k,odd}$ can be written as a linear combination of derivatives of Appell-Lerch sums. More precisely
\[\Lambda_{k,odd} =\frac 12\left(D_v^kA_1^{odd}\right)\left(0,\tau+\tfrac 12;2\tau\right),\]
where we define
\begin{align*}
A_1^{odd}(u,v;\tau):=&a^{1/2}\sum\limits_{\substack{ n\in\Z\\ n\text{ odd}}} \frac{(-1)^{n} q^{n(n+1)/2}b^n}{1-aq^n}\\ 
                       =&\frac 12\left(A_1(u,v;\tau)-A_1(u,v+\tfrac 12;\tau)\right),
                       \end{align*}
where again $a=e^{2\pi iu}$ and $b=e^{2\pi i v}$.                      
\end{lemma}
\begin{proof}
First we remark that the right-hand side of the identity to be shown is actually well-defined because as a function of $u$, $A_1(u,v;\tau)$ has simple poles in $\Z\tau+\Z$ (cf. \cite[Proposition 1.4]{ZwegersDiss}) which cancel out if the sum is only taken over odd integers. Thus the equation actually makes sense.

Then we write $\Lambda_{k,odd}$ as a $q$-series
\begin{align*}
2\Lambda_{k,odd}(\tau) 
= & \sum\limits_{\ell=0}^\infty\sum\limits_{m=0}^\infty \min(2\ell+1,2m+1)^k q^{(2\ell+1)(2m+1)} \\
= & 2\sum\limits_{\ell=0}^\infty \left((2\ell+1)^k \sum\limits_{r=1}^\infty q^{(2\ell+1)(2\ell+1+2r)}\right) +\sum\limits_{\ell=0}^\infty (2\ell+1)^k q^{(2\ell+1)^2} \\
= & 2\sum\limits_{\ell=0}^\infty (2\ell+1)^k q^{(2\ell+1)^2} \left(\frac{1}{1-q^{2(2\ell+1)}}-1\right)+\sum\limits_{\ell=0}^\infty (2\ell+1)^k q^{(2\ell+1)^2} \\
= & \sum\limits_{\ell=0}^\infty (2\ell+1)^k \left(\frac{q^{(2\ell+1)^2}+q^{(2\ell+1)^2+2(2\ell+1)}}{1-q^{2(2\ell+1)}}\right).
\end{align*}
This is easily seen to be the same as $\left(D_v^kA_1^{odd}\right)\left(0,\tau+\tfrac 12;2\tau\right)$.
\end{proof}
\begin{remark}
Now we can write down completions for each summand in \eqref{eqCohen2} and thus we see that the function
\begin{equation}\label{eqCohen3}
\frac{c_k}{2}\left([\G,\vartheta]_k(\tau)-[\G,\vartheta]_k\left(\tau+\tfrac 12\right)\right)+\left(D_v^{2k+1}\widehat{A}_1^{odd}\right)\left(0,\tau+\tfrac 12;2\tau\right)
\end{equation}
transforms like a modular form of weight $2k+2$. Because the Fourier coefficients of the holomorphic parts grow polynomially, they are holomorphic at the cusps as well. 

Thus it remains to show that the non-holomorphic parts given by
\begin{equation}\label{nonhol1}
\frac{c_k}{2}\left([\calR,\vartheta]_k(\tau)-[\calR,\vartheta]_k(\tau+\tfrac 12)\right)
\end{equation} 
and 
\begin{equation}\label{nonhol2}
\begin{aligned}
\frac{i}{4}&\left(\sum\limits_{\ell=0}^{2k+1} {{2k+1} \choose \ell} (-1)^\ell(D_u^\ell R)(-\tau-\tfrac 12;2\tau)(D_v^{2k-\ell+1}\Theta)(\tau+\tfrac 12;2\tau)\right.\\
 &-\left.\sum\limits_{\ell=0}^{2k+1} {{2k+1} \choose \ell} (-1)^\ell(D_u^\ell R)(-\tau-1;2\tau)(D_v^{2k-\ell+1}\Theta)(\tau+1;2\tau)\right)
 \end{aligned}
\end{equation}
are indeed equal up to sign and that the function in \eqref{eqCohen3} is modular on $\Gamma_0(4)$.
\end{remark}
This shows that we will need some specific information about the derivatives of the Jacobi theta series and the $R$-function evaluated at the torsion point $(\tau+\tfrac 12,2\tau)$.

A simple and straight forward calculation gives us the following result.
\begin{lemma}\label{diffTheta}
For $r\in\N_0$ one has
\begin{equation}\label{eqdiffTheta}
(D_v^r\Theta)(\tau+\tfrac 12;2\tau)=-q^{-1/4}\sum\limits_{s=0}^{\lfloor \tfrac r2\rfloor}{r \choose 2s } \left(-\frac 12\right)^{r-2s}(D_{\tau}^s\vartheta)(\tau)
\end{equation}
and
\begin{equation}
(D_v^r\Theta)(\tau+1;2\tau)=iq^{-1/4}\sum\limits_{s=0}^{\lfloor \tfrac r2\rfloor}{r \choose 2s } \left(-\frac 12\right)^{r-2s}(D_{\tau}^s\vartheta)(\tau+\tfrac 12),
\end{equation}
with $\Theta$ as in \eqref{Theta} and $\vartheta$ as in \eqref{theta}.
\end{lemma}
\begin{lemma}\label{Rid}
The following identities are true:
\begin{align}
\label{R1} R\left(-\tau-\tfrac 12;2\tau\right)&=iq^{1/4}\\
\label{R2} R(-\tau-1;2\tau)&=-q^{1/4}\\
\label{DR1}\left(D_uR\right)\left(-\tau-\tfrac 12;2\tau\right)&=\frac{-1+i}{4\pi}q^{1/4}\int\limits_{-\overline\tau}^{i\infty} \frac{\vartheta(z)}{(z+\tau)^{3/2}}d z -\frac i2 q^{1/4}\\
\label{DR2}\left(D_uR\right)(-\tau-1;2\tau)&=-\frac{1+i}{4\pi}q^{1/4}\int\limits_{-\overline\tau}^{i\infty} \frac{\vartheta(z+\frac 12)}{(z+\tau)^{3/2}}dz+\frac 12q^{1/4}.
\end{align}
\end{lemma}
\begin{proof}
The identities \eqref{R1} and \eqref{R2} follow directly by applying the transformation properties $(iii)$, $(i)$, and $(ii)$ of $R$ in \Cref{Rfunc}.

We only show \eqref{DR1}, since \eqref{DR2} then also follows from the obvious fact that $R(u;\tau+1)=e^{-\pi i/4}R(u;\tau)$. From the definition of $R$ in \eqref{R} and \eqref{E} we see that
\[(D_uR)(-\tau-\tfrac 12;2\tau)=iq^{1/4} \sum\limits_{n\in\Z} \left(\frac{1}{\sqrt{4y}\pi} e^{-4\pi n^2y}-\sign(n)(n+\frac 12)\beta(4n^2y)\right)q^{-n^2},\]
with $\beta$ as in \eqref{beta}. Note that for convenience, we define $\sign(0):=1$.

By partial integration one gets for all $x\in\R_{\geq 0}$ that
\[\beta(x)=\frac 1\pi x^{-1/2} e^{-\pi x} -\frac{1}{2\sqrt{\pi}}\Gamma(-\frac 12; \pi x),\]
where again, $\Gamma(\alpha;x)$ denotes the incomplete Gamma function. Using the well-known fact that for $\tau\in\bbH$ and $n\in\N$ it holds that
\[\int\limits_{-\overline\tau}^{i\infty} \frac{e^{2\pi inz}}{(-i(z+\tau))^{3/2}}d z=i(2\pi n)^{1/2}q^{-n}\Gamma(-\frac 12;4\pi ny)\]
we get the assertion by a straightforward calculation.
\end{proof}
Now we take a closer look at \eqref{nonhol2}.
\begin{lemma}\label{lemRTheta}
For all $k\in\N_0$ it holds true that
\begin{align*}
 & \sum\limits_{\ell=0}^{2k+1} (-1)^\ell {{2k+1} \choose \ell} (D_u^\ell R)(-\tau-\tfrac 12;2\tau)(D_v^{2k-\ell+1}\Theta)(\tau+\tfrac 12;2\tau)\\
=& q^{-1/4}\sum\limits_{m=0}^{k}\sum\limits_{\ell=0}^{k-m} \left[\frac 12 (D_u^{2\ell} R)(-\tau-\tfrac 12;2\tau)+\frac{2(k-\ell-m)+1}{2\ell+1}(D_u^{2\ell+1}R)(-\tau-\tfrac 12
;2\tau)\right]\\
 & \qquad\qquad\qquad \times b_{k,\ell,m}\left(\frac 12\right)^{2(k-\ell-m)}(D_\tau ^m\vartheta)(\tau)
\end{align*}
and
\begin{align*}
 & \sum\limits_{\ell=0}^{2k+1} (-1)^\ell {{2k+1} \choose \ell} (D_u^\ell R)(-\tau-1;2\tau)(D_v^{2k-\ell+1}\Theta)(\tau+1;2\tau)\\
=& -iq^{-1/4}\sum\limits_{m=0}^{k}\sum\limits_{\ell=0}^{k-m} \left[\frac 12 (D_u^{2\ell} R)(-\tau-1;2\tau)+\frac{2(k-\ell-m)+1}{2\ell+1}(D_u^{2\ell+1}R)(-\tau-;,2\tau)\right]\\
 &\qquad\qquad\qquad \times b_{k,\ell,m}\left(\frac 12\right)^{2(k-\ell-m)}(D_\tau ^m\vartheta)(\tau+\tfrac 12),
\end{align*}
where
\[b_{k,\ell,m}:=\frac{(2k+1)!}{(2\ell)!(2m)!(2(k-\ell-m)+1)!}={{2k+1} \choose {2\ell,\,2m,\,2(k-\ell-m)+1}}.\]
\end{lemma}
\begin{proof} 
Again, we only show the former equation, the latter follows from the transformation laws. For simplicity, we omit the arguments of the functions considered.

We obtain
\begin{align*}
 & \sum\limits_{\ell=0}^{2k+1} (-1)^\ell {{2k+1} \choose \ell} (D_u^\ell R)(D_v^{2k-\ell+1}\Theta)\\
=& \sum\limits_{\ell=0}^{k} {{2k+1} \choose {2\ell}} (D_u^{2\ell} R)(D_v^{2(k-\ell)+1}\Theta)-\sum\limits_{\ell=0}^{k} {{2k+1} \choose {2\ell+1}} (D_u^{2\ell+1} R)(D_v^{2(k-\ell)}\Theta)\\
\overset{\eqref{eqdiffTheta}}{=} & -q^{-1/4}\left[\sum\limits_{\ell=0}^{k}\sum\limits_{m=0}^{k-\ell} \underbrace{{{2k+1} \choose {2\ell}}{{2(k-\ell)+1} \choose {2m}}}_{=b_{k,\ell,m}}\left(-\frac 12\right)^{2(k-\ell-m)+1} (D_u^{2\ell} R)(D_\tau ^m\vartheta)\right.\\
& \qquad\qquad \left. -\sum\limits_{\ell=0}^{k}\sum\limits_{m=0}^{k-\ell} \underbrace{{{2k+1} \choose {2\ell+1}}{{2(k-\ell)} \choose {2m}}}_{=\frac{2(k-\ell-m)+1}{2\ell+1}b_{k,\ell,m}}\left(-\frac 12\right)^{2(k-\ell-m)} (D_u^{2\ell+1} R)(D_\tau ^m\vartheta)\right]\\
=&q^{-1/4}\sum\limits_{\ell=0}^{k}\sum\limits_{m=0}^{k-\ell} \left[\frac 12 (D_u^{2\ell} R)+\frac{2(k-\ell-m)+1}{2\ell+1}(D_u^{2\ell+1}R)\right]b_{k,\ell,m}\left(\frac 12\right)^{2(k-\ell-m)}(D_\tau ^m\vartheta).
\end{align*} 

Interchanging the sums gives the desired result.
\end{proof}
\begin{corollary}\label{corFin}
\Cref{conjCohen} is true if the identity
\begin{equation}\label{eqfin1}
\begin{aligned}
\left(D_\tau^m\calR\right)(\tau)=&-\frac i4 q^{-1/4} (-1)^m\sum\limits_{\ell=0}^m \left[\frac 12 (D_u^{2\ell}R)(-\tau-\tfrac 12;2\tau)\right.\\
                    &\qquad \left.+\frac{2(m-\ell)+1}{2\ell+1}(D_u^{2\ell+1}R)(-\tau-\tfrac 12;2\tau)\right]\cdot {{2m+1}\choose {2\ell}}\left(\frac{1}{4}\right)^{m-\ell}
\end{aligned}
\end{equation}
holds true for all $m\in\N_0$ and the function in \eqref{eqCohen3} is modular on $\Gamma_0(4)$.
\end{corollary}
\begin{proof}
\Cref{lemRTheta} gives us that \Cref{conjCohen} holds true if the identity
\begin{equation}
\begin{aligned}
\label{idcalR1} & c_k (-1)^{k-m} {{k+\tfrac 12} \choose m}{{k-\tfrac 12} \choose {k-m}} D^{k-m}_\tau\calR(\tau)\\
=& -\frac i4 q^{-1/4}\sum\limits_{\ell=0}^{k-m} \left[\frac 12 (D_u^{2\ell} R)(-\tau-\tfrac 12;2\tau)+\frac{2(k-\ell-m)+1}{2\ell+1}(D_u^{2\ell+1}R)(-\tau-\tfrac 12;2\tau)\right]\\
 &\qquad\qquad \times b_{k,\ell,m}\left(\frac 12\right)^{2(k-\ell-m)}
 \end{aligned}
 \end{equation}
does as well. 

We can simpify this a little further: We have
\begin{align*}
c_k{{k+\tfrac 12} \choose m}{{k-\tfrac 12} \choose {k-m}}={k \choose m} \frac{\sqrt\pi \Gamma(k+\tfrac 32)}{\Gamma(k-m+\tfrac 32)\Gamma(m+\tfrac 12)}
\end{align*}
and using Legendre's duplication formula for the Gamma function we obtain after a little calculation that
\begin{align*}
\frac{\left(\frac{1}{2}\right)^{2(k-\ell-m)}b_{k,\ell,m}}{c_k{{k+\tfrac 12} \choose m}{{k-\tfrac 12} \choose {k-m}}}= {{2(k-m)+1}\choose {2\ell}}\left(\frac{1}{4}\right)^{k-\ell-m}.
\end{align*}
and hence the corollary. 
\end{proof}
\begin{remark}\label{remLevel}
Since 
\[\begin{pmatrix} 1 & -\frac 12 \\ 0 & 1 \end{pmatrix}\Gamma_0(4)\begin{pmatrix} 1 & \frac 12 \\ 0 & 1 \end{pmatrix}=\Gamma_0(4)\]
we see that for any (not necessarily holomorphic) modular form $f$ of even weight $k$ on $\Gamma_0(4)$, the function $g=f|_k\begin{pmatrix} 1 & \frac 12 \\ 0 & 1 \end{pmatrix}$ is a modular form of the same weight on $\Gamma_0(4)$ as well. In particular, this applies to $[\G,\vartheta]_k$ for all $k\in\N_0$. 
\end{remark}
\begin{lemma}\label{lemLevel}
The second summand in \eqref{eqCohen3} transforms like a modular form on $\Gamma_0(4)$.
\end{lemma}
\begin{proof}
Looking at the $(2\ell+1)$-st derivative of $\widehat{A}_1^{odd}(0,v;\tau)$ with respect to $v$, one immediately sees that this has the modular transformation properties of a Jacobi form of weight $2\ell+2$ and index $0$ on $\SL_2(\Z)$. By \cite[Theorem 1.3]{EichZag} it follows that ${\mathcal{A}(\tau):=(D_v^{2\ell+1}\widehat{A}_1^{odd})(0,\tfrac{\tau}{2}+\tfrac 12;\tau)}$ transforms (up to some power of $q$) like a modular form of weight $\ell+1$ on the group
\[\Gamma:=\left\lbrace\gamma\in\SL_2(\Z)\,|\,\frac{a-1}{2}+\frac{c}{2}\in\Z\text{ and }\frac{b}{2}+\frac{d-1}{2}\in\Z\right\rbrace.\]
We are interested in $(D_v^{2\ell+1}\widehat{A}_1^{odd})(0,\tau+\tfrac 12;\tau)=\frac{1}{2^{\ell+1}}\mathcal{A}|_{2\ell+2}\begin{pmatrix} 2 & 0 \\ 0 & 1 \end{pmatrix}$ and since one easily checks that
\[\begin{pmatrix} 2 & 0 \\ 0 & 1 \end{pmatrix}\Gamma_0(4)\begin{pmatrix} \frac 12 & 0 \\ 0 & 1 \end{pmatrix}\leq \Gamma,\]
the assertion follows.
\end{proof}
\section{A Proof of Cohen's Conjecture}\label{secProof}
We now prove \Cref{main} using \Cref{corFin}. The proof is an induction on $m$. Since the base case $m=0$ gives an alternative proof of the class number relation \eqref{Eic} by Eichler, we give this as a proof of an additional theorem.
\begin{theorem}(M. Eichler, 1955, \cite{Eichler55})\\
For odd numbers $n\in\N$ we have the class number relation 
\[\sum\limits_{s\in\Z}H(n-s^2)+\lambda_1(n)=\frac 13\sigma_1(n).\]
\end{theorem}
\begin{proof}
Let
\[F_2(\tau)=\sum\limits_{n=0}^\infty \sigma_1(2n+1)q^{2n+1}.\]
We recall that this function is a modular form of weight $2$ on $\Gamma_0(4)$ (cf. eg. \cite[Proposition 1.1]{Coh75}). 

Plugging in $m=0$ into \eqref{eqfin1} gives us the equation
\[\calR(\tau)=-\frac i4 q^{-1/4} \left[\frac 12 R(-\tau-\tfrac 12;2\tau)+(D_uR)(-\tau-\tfrac 12;2\tau)\right].\]
This equality holds true by \Cref{Rid}. Hence we know by \Cref{corFin}, \Cref{lemAppell}, \Cref{remLevel} and \Cref{lemLevel} that 
\[\frac 12\left(\calH(\tau)\vartheta(\tau)-\calH\left(\tau+\tfrac 12\right)\vartheta\left(\tau+\tfrac 12\right)\right)+\Lambda_{1,odd}(\tau)\]
is indeed a holomorphic modular form of weight $2$ on $\Gamma_0(4)$ as well.

Since the space of modular forms of weight $2$ on $\Gamma_0(4)$ is $2$-dimensional, the assertion follows by comparing the first two Fourier coefficients of the function above and $\tfrac 13 F_2(\tau)$.
\end{proof}
The proof of this given in \cite{Eichler55} involves topological arguments about the action of Hecke operators on the Riemann surface associated to $\Gamma_0(2)$ on the one hand and arithmetic of quaternion orders on the other.

Note that the knowledge of Eichler's class number relation does not necessarily imply that the non-holomorphic parts of our considered mixed mock modular forms cancel. 

\begin{proof}[Proof of \Cref{main}]
The base case of our induction is treated above, thus suppose that \eqref{eqfin1} holds true for one $m\in\N_0$.

For simplicity, we omit again the argument $\left(-\tau-\tfrac 12;2\tau\right)$ in the occuring $R$ derivatives.

By the induction hypothesis we see that
{\allowdisplaybreaks
\begin{align*}
 & D_\tau^{m+1}\calR(\tau)=D_\tau(D_\tau^m\calR(\tau))\\
=& -\frac i4 q^{-1/4} (-1)^m \left\lbrace -\frac 14 \sum\limits_{\ell=0}^m \left[\frac 12 (D_u^{2\ell}R)+\frac{2(m-\ell)+1}{2\ell+1}(D_u^{2\ell+1}R)\right]\cdot {{2m+1}\choose {2\ell}}\left(\frac{1}{4}\right)^{m-\ell} \right.\\
 & \qquad +\left. \sum\limits_{\ell=0}^m \left[\frac 12 (D_\tau D_u^{2\ell}R)+\frac{2(m-\ell)+1}{2\ell+1}(D_\tau D_u^{2\ell+1}R)\right]\cdot {{2m+1}\choose {2\ell}}\left(\frac{1}{4}\right)^{m-\ell}\right\rbrace .\\
\end{align*}
}
By the Theorem of Schwarz, the partial derivatives interchange and thus the total differential $D_\tau$ is given by
\[D_\tau(D_u^\ell R(-\tau-\tfrac 12;2\tau))=-(D_u^{\ell+1}R)(-\tau-\tfrac 12;2\tau)+2(D_\tau D_u^\ell R)(-\tau-\tfrac 12;2\tau).\]
Now \Cref{Heat} implies that the above equals
{\allowdisplaybreaks
\begin{align*}
 & -\frac i4 q^{-1/4} (-1)^{m+1} \left\lbrace \sum\limits_{\ell=0}^m \left[\frac 18 (D_u^{2\ell}R)+\left(\frac 14 \frac{2(m-\ell)+1}{2\ell+1}+\frac 12\right)(D_u^{2\ell+1}R)\right.\right.\\
 & \qquad\qquad \left.\left. +\left(\frac 12+\frac{2(m-\ell)+1}{2\ell+1}\right)(D_u^{2\ell+2}R) +\frac{2(m-\ell)+1}{2\ell+1}(D_u^{2\ell+3}R)\right]\cdot {{2m+1}\choose {2\ell}}\left(\frac{1}{4}\right)^{m-\ell} \right\rbrace \\
=& -\frac i4 q^{-1/4} (-1)^{m+1} \left\lbrace \left[\frac{1}{2}R +(2m+3)(D_uR)\right]\left(\frac 14\right)^{m+1} \right.\\
 & +\sum\limits_{\ell=1}^{m} {{2m+1}\choose {2\ell}}\left(\frac{1}{4}\right)^{m-\ell+1}\left[\left(\frac 12+\frac{2(2m-\ell)+5}{4\ell-2}\cdot\frac{(2\ell-1)(2\ell)}{(2(m-\ell)+2)(2(m-\ell)+3)}\right)(D_u^{2\ell}R)\right.\\
 & \qquad \quad \left.+\left(\frac{2(m+\ell)+3}{2\ell+1}+\frac{2(m-\ell)+3}{2\ell-1}\cdot\frac{(2\ell-1)(2\ell)}{(2(m-\ell)+2)(2(m-\ell)+3)}\right)(D_u^{2\ell+1}R)\right] \\
 &  + \left.\left[\frac{2m+3}{4m+2}(D_u^{2m+2}R)+\frac{1}{2m+1}(D_u^{2m+3}R)\right]{{2m+1} \choose {2m}}\right\rbrace.
\end{align*}
}
It is easily seen that the last summand equals
\[\left[\frac 12(D_u^{2m+2}R)+\frac{1}{2m+3}(D_u^{2m+3}R)\right]\cdot (2m+3)\]
and a direct but rather tedious calculation gives that
\begin{align*}
 &{{2m+1}\choose {2\ell}}\cdot\left(\frac 12+\frac{2(2m-\ell)+5}{4\ell-2}\cdot\frac{(2\ell-1)(2\ell)}{(2(m-\ell)+2)(2(m-\ell)+3)}\right)=\frac 12 {{2m+3}\choose {2\ell}}
\end{align*}
and 
\begin{align*}
 &{{2m+1}\choose {2\ell}}\cdot\left(\frac{2(m+\ell)+3}{2\ell+1}+\frac{2(m-\ell)+3}{2\ell-1}\cdot\frac{(2\ell-1)(2\ell)}{(2(m-\ell)+2)(2(m-\ell)+3)}\right)\\
 = &\frac{2(m-\ell)+3}{2\ell+1} {{2m+3}\choose {2\ell}}.
\end{align*}
In summary, we therefore get
\begin{align*}
D_\tau^{m+1}\calR(\tau)=&-\frac i4 q^{-1/4} (-1)^{m+1}\sum\limits_{\ell=0}^{m+1} \left[\frac 12 (D_u^{2\ell}R)(-\tau-\tfrac 12;2\tau)\right.\\
                        &\qquad \left. +\frac{2(m-\ell)+3}{2\ell+1}(D_u^{2\ell+1}R)(-\tau-\tfrac 12;2\tau)\right]\cdot {{2m+3}\choose {2\ell}}\left(\frac{1}{4}\right)^{m-\ell+1}
\end{align*}
which proves \Cref{conjCohen}.

The fact that we actually get a cusp form can be seen in the following way: 

By \cite[Corollary 7.2]{Coh75} we see that the function $\tau\mapsto\frac{c_k}{2}\left([\calH,\vartheta]_k(\tau)-[\calH,\vartheta]_k(\tau+\tfrac 12)\right)$ is a non-holomorphic cusp form if $k\geq 1$. We use the same argument as there to see that $(D_u^{2\ell+1}\widehat{A}_1^{odd})(0,\tau+\tfrac 12;2\tau)$ is a cusp form as well. Because we know by \Cref{lemLevel} that we have for $\gamma\in\SL_2(\R)$ that
\[(D_v^{2\ell+1}\widehat{A}_1^{odd})|_{2\ell+2}\gamma=D_v^{2\ell+1}(\widehat{A}_1^{odd}|_{1}\gamma),\]
and by definition $(D_v^{2\ell+1}\widehat{A}_1^{odd})(0,\tau+\tfrac 12; 2\tau)$ vanishes at the cusp $i\infty$ for all $\ell\in\N_0$. So by the above equation it vanishes at every cusp of $\Gamma_0(4)$. 
\end{proof}
\begin{corollary}\label{corClass}
By comparing the first few Fourier coefficients of the modular forms in \Cref{main} one finds for all odd $n\in\N$ the following class number relations
\begin{align*}
\sum\limits_{s\in\Z} \left(4s^2-n\right)H\left(n-s^2\right)+\lambda_3\left(n\right)&=0,\\
\sum\limits_{s\in\Z} g_4(s,n)H\left(n-s^2\right)+\lambda_5\left(n\right)&=-\frac{1}{12}\sum\limits_{n=x^2+y^2+z^2+t^2}\mathcal{Y}_4(x,y,z,t),\\
\sum\limits_{s\in\Z} g_6(s,n)H\left(n-s^2\right)+\lambda_7\left(n\right)&=-\frac 13\sum\limits_{n=x^2+y^2+z^2+t^2} \mathcal{Y}_6(x,y,z,t),\\
\sum\limits_{s\in\Z} g_8(s,n)H\left(n-s^2\right)+\lambda_9\left(n\right)&=-\frac{1}{70}\sum\limits_{n=x^2+y^2+z^2+t^2} \mathcal{Y}_8(x,y,z,t)
\end{align*}
where $g_\ell(n,s)$ is the $\ell$-th Taylor coefficient of $\left(1-sX+nX^2\right)^{-1}$ and $\mathcal{Y}_d(x,y,z,t)$ is a certain harmonic polynomial of degree $d$ in $4$ variables. Explicitly, we have 
\begin{align*}
g_4(s,n)&=\left(16s^4-12ns^2+n^2\right),\\
g_6(s,n)&=\left(64s^6-80s^4n+24s^2n^2-n^3\right),\\
g_8(s,n)&=\left(256s^8 - 448s^6n + 240s^4n^2 - 40s^2n^3 + n^4\right),
\end{align*}  
and
\begin{align*}
\mathcal{Y}_4(x,y,z,t)&=\left(x^4-6x^2y^2+y^4\right),\\
\mathcal{Y}_6(x,y,z,t)&= \left(x^6 - 5x^4y^2 - 10x^4z^2 + 30x^2y^2z^2  + 5x^2z^4 - 5y^2z^4\right),\\
\mathcal{Y}_8(x,y,z,t)&= \left(13 x^8 + 63 x^6 y^2 - 490 x^6 z^2 + 63 x^6 t^2 - 630 x^4 y^2 z^2  - 315 x^4 y^2 t^2 + 1435 x^4 z^4 \right.\\
&\quad\left. - 630 x^4 z^2 t^2  + 315 x^2 y^2 z^4 + 1890 x^2 y^2 z^2 t^2  - 616 x^2 z^6 + 315 x^2 z^4 t^2 -315 t^2 y^2 z^4  + 22 z^8     \right).\\
\end{align*}
The first two of the above relations were already mentioned in \cite{Coh75}.
\end{corollary}
\begin{remark}
The formula \eqref{eqCohen} looks indeed very similar to the Eichler-Selberg trace formula as given in \cite{Coh75}, so one might ask whether our result gives a similar trace formula for Hecke operators on the subspace $\mathcal{S}^{odd}_k(\Gamma_0(4))$ of cusp forms of weight $k$ on $\Gamma_0(4)$ with only odd $q$-powers in their Fourier expansion (this space is of course Hecke invariant). Unfortunately, computer experiments showed that this is in fact \emph{not} the case: As soon as $\dim\mathcal{S}^{odd}_k(\Gamma_0(4))\geq 2$, i.e. $k\geq 10$, the cusp form we get is \emph{not} a multiple of the generating function of traces of Hecke operators.
\end{remark}
\section*{Acknowledgements}
The author would like to thank Prof. Dr. Kathrin Bringmann for suggesting this topic as part of his PhD thesis \cite{Mert14}. He also thanks his colleagues at the Universt\"at zu K\"oln, especially Dr. Ben Kane, Maryna Viazovska, and Ren\'e Olivetto, for many fruit- and helpful discussions.
\providecommand{\bysame}{\leavevmode\hbox to3em{\hrulefill}\thinspace}
\providecommand{\MR}{\relax\ifhmode\unskip\space\fi MR }
\providecommand{\MRhref}[2]{%
  \href{http://www.ams.org/mathscinet-getitem?mr=#1}{#2}
}
\providecommand{\href}[2]{#2}

\end{document}